\documentclass[10pt, twoside, leqno]{article}
\usepackage[indentfirst]{titlesec}
\usepackage{amssymb,latexsym}
\usepackage{enumerate}
\usepackage{amsmath, amsthm, amsfonts, amssymb, mathrsfs}
\newtheorem{theorem}{Theorem}[section]
\newtheorem{lemma}[theorem]{Lemma}

\newtheorem{definition}[theorem]{Definition}
\newtheorem{remark}[theorem]{Remark}
\newtheorem{conjecture}[theorem]{Conjecture}

\numberwithin{equation}{section}

\numberwithin{equation}{section}

\begin{document}

\baselineskip=17pt

\title{Variational characterization of  $H^p$}

\date{}

\maketitle

 \begin{center}

{\bf Honghai Liu}\\
School of Mathematics and information Science\\
Henan Polytechnic University\\
Jiaozuo 454003\\
The People's Republic of China \\
E-mail: {\it hhliu@hpu.edu.cn}
\end{center}

\renewcommand{\thefootnote}{}



\footnote{The research was supported by
NSF of China (Grant: 11501169 and 11371057).}

\renewcommand{\thefootnote}{\arabic{footnote}}
\setcounter{footnote}{0}

\begin{abstract}
In this paper we obtain the variational characterization of Hardy space $H^p$ for $p\in(\frac n{n+1},1]$ and get estimates for the oscillation operator and the $\lambda$-jump operator associated with approximate identities acting on $H^p$ for $p\in(\frac n{n+1},1]$. Moreover, we give counterexamples to show that the oscillation and $\lambda$-jump associated with some approximate identity can not be used to characterize $H^p$ for $p\in(\frac n{n+1},1]$.
\end{abstract}
\newpage

 \section{Introduction}
 Variational, oscillation and jump inequalities have been the subject of many
 recent articles in probability, ergodic theory and harmonic
 analysis. The first variational inequality was proved by L\'epingle \cite{Lep76} for martingales.  Using L\'epingle's result, Bourgain \cite{Bou89} is the first one who to obtain corresponding variational estimates for the Birkhoff ergodic averages and then directly deduce pointwise convergence results without previous knowledge that pointwise convergence holds for a dense subclass of functions, which are not available in some ergodic models. Bourgain's work has initiated a new research direction in ergodic theory and harmonic analysis. In \cite{CJRW2000,CJRW2002,JKRW98,JRW03,JSW08}, Jones and his collaborators systematically studied jump and variational inequalities for ergodic averages and truncated singular integrals. Since then many other publications came to enrich the literature on this subject. Analogs are also true for corresponding maximal operators which were known. A number of phenomena show that variational, jump and oscillation operators seems to play the same role as maximal operators in harmonic analysis.
 \par
 In this paper we consider variation, oscillation and $\lambda$-jump operators associated
with approximate identities. The variational inequality gives us the characterization of $H^p$ for $p\in(\frac n{n+1},1]$. We obtain estimates for the oscillation and $\lambda$-jump acting on Hardy space. A counterexample show that oscillation and $\lambda$-jump operators are too small to characterize Hardy space. Before we present our main results we recall some definitions and known results.

Let $\mathcal I$ be a subset of $\mathbb R^+$, $\mathfrak{a}=\{a_t: t\in \mathcal I\}$ be a family of complex numbers and $\rho\ge1$. The $\rho$-variation norm of the family $\mathfrak{a}$ is defined by
\begin{equation}\nonumber
\|\mathfrak{a}\|_{v_\rho(\mathcal I)}=\sup\big(\sum_{k\geq1}
|a_{t_k}-a_{t_{k-1}}|^\rho\big)^{\frac{1}{\rho}},
\end{equation}
where the supremum runs over all finite decreasing sequences $\{t_k\}$ in $\mathcal I$. We denote the norm $v_\rho(\mathbb R^+)$ by $v_\rho$ for short. It is trivial that
\begin{equation}\nonumber
\|\mathfrak{a}\|_{L^\infty(\mathcal I)}:=\sup_{t\in \mathcal I}|a_t|
\le |a_{t_0}|+\|\mathfrak{a}\|_{v_\rho(\mathcal I)}
\end{equation}
for any $t_0\in \mathcal I$ and $\rho\ge1$.
\par
Given a family of Lebesgue measurable functions $\mathcal F(x)=\{F_t(x):t\in \mathcal I\}$, the value of the $\rho$-variation function $\mathscr V_q(\mathcal F)$ of the family $\mathcal F$ at $x$ is defined by
\begin{equation}\nonumber
\mathscr V_\rho(\mathcal F)(x)=\|\{F_t(x)\}\|_{v_\rho(\mathcal I)},\quad \rho\ge1.
\end{equation}
 Specially, suppose $\mathscr{A}=\{{A}_t\}_{t>0}$ is a family of operators, the $\rho$-variation operator related $\mathscr A$ is simply defined as
$$\mathscr V_\rho(\mathscr Af)(x)=\|\{A_t(f)(x)\}_{t>0}\|_{v_\rho}.$$
It is easy to observe that for any fixed $x\in\mathbb R^n$, if $\mathscr V_\rho(\mathscr Af)(x)<\infty$, then $\lim\limits_{t\rightarrow0^+}A_t(f)(x)$ and  $\lim\limits_{t\rightarrow+\infty}A_t(f)(x)$ exist. In particular, if $\mathscr V_\rho(\mathscr Af)$ belongs to some function spaces such as $L^p$ or $L^{p,\infty}$, then the sequence converges almost everywhere without any additional condition. This is why mapping property of $\rho$-variation operator is so interesting in probability, ergodic theory and harmonic analysis. In 1976, L\'{e}pingle \cite{Lep76} showed that the $\rho$-variation operator related to a bounded  martingale sequence is a bounded operator on $L^p$ for $1<p<\infty$ and $\rho>2$. These estimates can fail for $\rho\le2$, see \cite{JW04,Q98}. So, we need the oscillation operator to substitute the $2$-variation operator.
\par
For each fixed decreasing sequence $\{t_i\}$ in $\mathbb R^+$, we also define the oscillation operator related to $\mathscr A$
\begin{equation}\nonumber
\mathscr O(\mathscr Af)(x)=\big(\sum_{i=1}^\infty\sup_{t_{i}\le\varepsilon_{i}<\varepsilon_{i+1}\le t_{i+1}}|A_{\varepsilon_{i+1}}f(x)-A_{\varepsilon_{i}}f(x)|^2\big)^{1/2}.
\end{equation}
\par
We also study the $\lambda$-jump operator. For $\lambda>0$, the value of the $\lambda$-jump function for $\mathcal F$ at $x$ is defined by
$$N_\lambda(\mathcal F)(x)=\sup\big\{N\in\Bbb N:\ \exists\ s_1<\varepsilon_1\leq s_2<\varepsilon_2\leq\dotsc\leq s_N<\varepsilon_N\ \text{such that\ } |F_{\varepsilon_k}(x)-F_{s_k}(x)|>\lambda\big\}.$$
Similarly, we define the $\lambda$-jump operator related to $\mathscr A$ as
$N_\lambda(\mathscr Af)(x)=N_\lambda(\{A_tf\}_{t>0})(x)$. Obviously, if $\lim\limits_{t\rightarrow0^+}A_tf(x)$ and $\lim\limits_{t\rightarrow+\infty}A_tf(x)$ exist, then $N_\lambda(\mathscr Af)(x)<\infty$ for any $\lambda>0$. Moreover, for $\lambda>0$ and $\rho\ge1$
 \begin{equation}\label{contr ineq}
\lambda[N_\lambda(\mathscr Af)(x)]^{1/\rho}\leq C_\rho\mathscr V_\rho(\mathscr Af)(x).
\end{equation}
\par
Let $\phi\in\mathscr S$ with $\int \phi dx=1$, $\phi_t(x)=\frac1{t^n}\phi(\frac xt)$, denote function family $\{\phi_t\ast f(x)\}_{t>0}$ by $\Phi\star f(x)$.  Let $f$ be a tempered distribution, we define maximal function $M_\phi$ by
$$
M_\phi f(x)=\sup_{t>0}|(f\ast\phi_t)(x)|.
$$
\begin{definition}
Let $0<p<\infty$. A distribution $f$ belongs to $H^p$ if the maximal function $M_\phi f$ is in $L^p$.
\end{definition}
\par
The main results of this paper are the following three theorems.
\begin{theorem}\label{SchHp}
For any $\rho>2$, there exists $C_{\rho}>0$ such that
\begin{equation}\label{VISchHp}
\|\mathscr V_\rho(\Phi\star f)\|_{L^p}\le C_{\rho}\|f\|_{H^p},\ \ \frac n{n+1}<p\le1.
\end{equation}
Moreover, for $\frac n{n+1}<p<\infty$, the following conditions are equivalent:
\begin{description}
\item (i) There is a $\phi\in\mathscr S$ with $\int\phi dx\neq0$ so that $M_\phi f\in L^p$.
\item (ii) For any $\rho>2$, there is a $\phi\in\mathscr S$ with $\int\phi dx\neq0$ so that $\phi\ast f$ and $\mathscr V_\rho(\Phi\star f)$ in $L^p$.
\end{description}
\end{theorem}
In the above result, the variation operator is used to characterize $H^p$ spaces,
it is natural to ask if the analogue for the oscillation operator holds.
\begin{theorem}\label{oHp}
For any $\{t_i\}\searrow0$ and $p\in(\frac n{n+1},1]$, there exists a positive constant $C_p$ such that
\begin{equation}\label{OISchHp}
\|\mathscr O(\Phi\star f)\|_{L^p}\le C_p\|f\|_{H^p}.
\end{equation}
Moreover, there exists $\{t_i\}\searrow0$, $\phi\in\mathscr S$ with $\int\phi dx\neq0$ and $f\in \mathscr S$ such that $\mathscr O(\Phi\star f)\in L^p$ and $M_\phi(f)\notin L^p$ for any $p\in(0,1]$.
\end{theorem}
We also apply the above result on variation to provide estimates
for the $\lambda$-jump operator.
\begin{theorem}\label{ljrhp}
If $\rho>2$, then the $\lambda$-jump operator $N_\lambda(\Phi\star f)$ satisfies
\begin{equation}
\|\lambda N_\lambda(\Phi\star f)^{1/\rho}\|_{L^p}\le C_{\rho}\|f\|_{H^p},\ \ \frac n{n+1}<p\le1,
\end{equation}
uniformly in $\lambda>0$. Moreover, there exists $\phi\in\mathscr S$ with $\int\phi dx\neq0$ and $f\in \mathscr S$ such that $\|\lambda N_\lambda(\Phi\star f)^{1/\rho}\|_{L^1}<\infty$ uniformly in $\lambda>0$ and $M_\phi(f)\notin L^1$.
\end{theorem}

The paper is organized as follows. In Section 2, we prove that $\rho$-variation, oscillation and $\lambda$-jump operators related to approximate identities are of strong type $(p,p)$ for $1<p<\infty$ and weak type $(1,1)$. Using the strong $L^p(p>1)$ estimates for the $\rho$-variation operator and the atom decomposition of Hardy space, we show Theorem \ref{SchHp} in Section 3. We consider the estimate for the oscillation operator associated to approximate identities acting on Hardy space and show that the oscillation is not proper to characterize Hardy spaces in Section 4. Finally, we present a jump inequality for the $\lambda$-jump operator on Hardy space and conjecture that its improvement holds, we also illustrate that the $\lambda$-jump operator can not be used to characterize $H^1$.
\section{Operators related to approximate identities on $L^p$}
To study the mapping property of variation, oscillation and $\lambda$-jump operators related to approximate identities on $H^p$, we need strong $L^p$ estimates for $1<p<\infty$. We use the argument in \cite{JSW08} and include all details for completeness although they are trivial. The following lemmas will be used later.
\begin{lemma}(\cite[Lemma 1.3]{JSW08})\label{pdN} Let $\mathscr A=\{A_t\}_{t>0}$ be a family of operators. Then
$$
\lambda\sqrt{N_\lambda(\mathscr A f)}\le C[S_2(\mathscr A f)+\lambda \sqrt{N^d_{\lambda/3}(\mathscr A f)}],
$$
where $N^d_\lambda(\mathscr A f)=N_\lambda(\{A_{2^k}f\})$ and $S_2(\mathscr Af)=(\sum_j\|\mathscr Af\|_{v_2(2^j,2^{j+1}]}^2)^{1/2}$.
\end{lemma}
Let $\sigma$ be a compactly supported finite Borel measure and satisfying
\begin{equation}\label{fts}
|\hat{\sigma}(\xi)|\le C|\xi|^{-b}, \ \ for\ some\ b>0.
\end{equation}
 $\sigma_t$ is given by $<\sigma_t,f>=\int f(tx)d\sigma$.
\begin{lemma}(\cite[Theorem 1.1,\ Lemma 6.1]{JSW08})\label{NdS2}
Let $\mathfrak U=\{\mathfrak U_k\}$ where $\mathfrak U_kf=f\ast\sigma_{2^k}$. If $\sigma$ satisfies \eqref{fts}, then
$$
\|\lambda\sqrt{N_\lambda(\mathfrak U f)}\|_{L^p}\le C_p\|f\|_{L^p}
$$
uniformly in $\lambda>0$. Moreover, let $\mathscr A=\{A_t\}_{t>0}$ where $A_tf=f\ast\sigma_t$. If $\sigma$ satisfies \eqref{fts} for some $b>1/2$, then
$$
\|S_2(\mathscr Af)\|_{L^p}\le C_p\|f\|_{L^p}
$$
holds for $\min\{2n/(n+2b-1),(2b+1)/2b\}<p<\max\{2n/(n-2b-1),2b+1\}$.
\end{lemma}
\begin{lemma}(\cite[Theorem 1.1]{JSW08})\label{Ndw11}
If $\phi$ satisfies
\begin{equation}
\int_{\mathbb R^n}|\phi(x+y)-\phi(x)|dx\le C|y|^{-b}
\end{equation}
for some $b>0$, then for any $\alpha>0$ we have
\begin{equation}
|\{x:\lambda\sqrt{N_\lambda^d(\Phi\star f)(x)}>\alpha\}|\le \frac C\alpha\|f\|_{L^1}
\end{equation}
uniformly in $\lambda>0$.
\end{lemma}
 We now state a theorem on the $\lambda$-jump operator associated with approximate identities.
\begin{theorem}\label{ljphip}
Let $\phi\in\mathscr S$ with $\int\phi dx\neq0$. For $1<p<\infty$, there exists a positive constant $C_p$ such that
\begin{equation}\label{JISchp}
\|\lambda\sqrt{N_\lambda(\Phi\star f)}\|_{L^p}\le C_p\|f\|_{L^p},
\end{equation}
 uniformly in $\lambda>0$. Moreover, for any $\alpha>0$,
\begin{equation}\label{JISch1}
|\{x:\lambda\sqrt{N_\lambda(\Phi\star f)(x)}>\alpha\}|\le \frac C\alpha\|f\|_{L^1},
\end{equation}
uniformly in $\lambda>0$.
\end{theorem}
\begin{proof}
Clearly,
$\hat{\phi}\in\mathscr S$, which means that for any $N\in\mathbb N$ there is a $C_N$ such that
$$
|\hat{\phi}(\xi)|\le C_{N}(1+|\xi|)^{-N}.
$$
Therefore $|\hat{\phi}(\xi)|\le C_{n}|\xi|^{-\frac{n+1}2}$. By Lemma \ref{NdS2}, we have
\begin{equation}
\|\lambda\sqrt{N_\lambda^d(\Phi\star f)}\|_{L^p}\le C_p\|f\|_{L^p}\ \ \text{and}\ \ \|S_2(\Phi\star f)\|_{L^p}\le C_p\|f\|_{L^p},\ \ 1<p<\infty.
\end{equation}
By Lemma \ref{pdN}, we get \eqref{JISchp}. We turn to the proof of \eqref{JISch1}. Lemma \ref{pdN}  and Lemma \ref{Ndw11} imply that it suffices to prove
\begin{align}\label{S2w1}
|\{x:S_2(\Phi\star f)(x)>1\}|\le C\|f\|_{L^1}.
\end{align}
We perform the Calder\'{o}n-Zygmund decomposition of $f$ at height $1$ and write $f=g+b$. We just need to establish the following two estimates:
\begin{equation}\label{sg11}
|\{x:S_2(\Phi\star g)(x)>1/2\}|\le C
\|f\|_{L^1},
\end{equation}
and
\begin{equation}\label{sb11}
|\{x:S_2(\Phi\star b)(x)>1/2\}|\le C
\|f\|_{L^1}.
\end{equation}
 As usual the known $L^p$ bounds for $S_2$ allows us to obtain (\ref{sg11}):
\begin{align*}
|\{x:S_2(\Phi\star g)(x)>1/2\}|&\le C\int_{\mathbb R^n}|S_2(\Phi\star g)(x)|^2dx\\
&\le C \int_{\mathbb R^n}|g(x)|^2dx\le C\|f\|_{L^1}.
\end{align*}
\par
To show \eqref{sb11}, we write $b=\sum_jb_j$ precisely, where each $b_j$ is supported in a dyadic cube $Q_j$. We denote by $l(Q_j)$ the side length of $Q_j$.
Let $\tilde{Q}_j$ be the cube with sides parallel to the axes having the same center as $Q_j$ and having side length $8l(Q_j)$, write $\tilde{Q}=\bigcup \tilde{Q}_j$. Obviously,
\begin{align*}
|\tilde{Q}|\le\sum_j |\tilde{Q}_j|\le C\sum_j |Q_j|\le C\|f\|_{L^1}.
\end{align*}
We still need to prove that
\begin{align*}\label{wsbc11}
|\{x\in\tilde{Q}^c:S_2(\Phi\star b)(x)>1/2\}|\le C\|f\|_{L^1}.
\end{align*}
Let $x_j$ be the center of $Q_j$. Using the moment condition of $b_j$ and H\"{o}lder's inequality, we get
\begin{align*}
\|\Phi\star b(x)\|_{v_2(2^i,2^{i+1}]}\le\sum_j\int_{Q_j}|b_j(y)|\|\Phi(x-y)-\Phi(x-x_j)\|_{v_2(2^i,2^{i+1}]}dy.
\end{align*}
For $x\in \tilde{Q}^c$ and $x_j,y\in Q_j$, it is clear that $|x-x_j-\theta (y-x_j)|\sim |x-x_j|$ for any $\theta\in [0,1]$. Then,
\begin{align*}
\|\Phi(x-y)-\Phi(x-x_j)\|_{v_2(2^i,2^{i+1}]}
&\le \|\Phi(x-y)-\Phi(x-x_j)\|_{v_1(2^i,2^{i+1}]}\\
&\le \int_{2^i}^{2^{i+1}}|\frac d{dt}[\phi_t(x-y)-\phi_t(x-x_j)]|dt\\
&\le C|y-x_j|\int_{2^i}^{2^{i+1}}\frac{1}{t^{n+2}}\big(1+\frac{|x-x_j|}t)^{-(n+2)}dt,
\end{align*}
where we use the following well-known fact
\begin{equation}\label{vr1i}
\|\mathfrak a\|_{v_\rho}\le \|\mathfrak a\|_{v_1}\le \int_0^\infty|\mathfrak a'(t)|dt,
\end{equation}
see (39) in \cite{JSW08}. Consequently,
\begin{equation}\label{eoV2i}
\|\Phi(x-y)-\Phi(x-x_j)\|_{v_2(2^i,2^{i+1}]}\le Cl(Q_j)\min\{2^{i+1}|x-x_j|^{-n-2}, 2^{-i(n+1)}\}.
\end{equation}
Finally, by \eqref{eoV2i} and H\"{o}lder's inequality,
\begin{align*}
|\{x\in\tilde{Q}^c:S_2(\Phi\star b)(x)>1/2\}|&\le C
\int_{\tilde{Q}^c}(\sum_{i\in\mathbb Z}\|\Phi\star b(x)\|_{v_2(2^i,2^{i+1}]}^2)^{\frac12}dx\\
&\le C\int_{\tilde{Q}^c}\sum_j\int_{Q_j}|b_j(y)|[\sum_{i\in\mathbb Z}\|\Phi(x-y)-\Phi(x-x_j)\|_{v_2(2^i,2^{i+1}]}^2]^{\frac12}dydx\\
&\le C\sum_jl(Q_j)\int_{Q_j}|b_j(y)|dy\int_{\tilde{Q_j}^c}\frac{dx}{|x-x_j|^{n+1}}\\
&\le C\|f\|_{L^1}.
\end{align*}
This completes the proof of Theorem \ref{ljphip}.
\end{proof}
To obtain the variational inequality, we present a lemma reducing variational inequalities to jump inequalities, which is a generalization of Bourgain's argument in \cite{Bou89}.
\begin{lemma}(\cite[Lemma 2.1]{JSW08})\label{jcv}
Suppose that $p_0<q<p_1$ and that for $p_0<p<p_1$ the inequality
$$
\sup_{\lambda>0}\|\lambda[N_\lambda(\mathcal Tf)]^{1/q}\|_{L^p}\le C\|f\|_{L^p}
$$
holds for all $f$ in $L^p$. Then we have for $q<\rho$,
$$
\|\mathscr V_\rho(\mathcal Tf)\|_{L^p}\le C(p,\rho)\|f\|_{L^p}
$$
for $f\in L^p$, $p_0<p<p_1$.
\end{lemma}
As a result of above lemma, the following variational inequality holds:
\begin{theorem}\label{SchLp}
For any $\rho>2$ and $1<p<\infty$, there exists $C_{p,\rho}>0$ such that
\begin{equation}\label{VISchp}
\|\mathscr V_\rho(\Phi\star f)\|_{L^p}\le C_{p,\rho}\|f\|_{L^p}.
\end{equation}
Moreover, for any $\alpha>0$,
\begin{equation}\label{VISch1}
|\{x:\mathscr V_\rho(\Phi\star f)(x)>\alpha\}|\le \frac C{\alpha}\|f\|_{L^1}.
\end{equation}
\end{theorem}
\begin{proof}
Clearly, \eqref{JISchp} and Lemma \ref{jcv} imply \eqref{VISchp}. \eqref{VISch1} can be proved as \eqref{S2w1}.
\end{proof}
\begin{theorem}\label{OSchLp}
For any $\{t_i\}\searrow0$ and $1<p<\infty$, there exists $C_p>0$ such that
\begin{equation}\label{OISchp}
\|\mathscr O(\Phi\star f)\|_{L^p}\le C_{p}\|f\|_{L^p}.
\end{equation}
Moreover, for any $\alpha>0$,
\begin{equation}\label{OISch1}
|\{x:\mathscr O(\Phi\star f)(x)>\alpha\}|\le \frac C{\alpha}\|f\|_{L^1}.
\end{equation}
\end{theorem}
\begin{proof} Let $k_i$ be the smallest integer such that $2^{k_i}$ greater than or equal to $t_i$. The long oscillation operator is given by
\begin{equation}\nonumber
\mathscr O_L(\Phi\star f)(x)=\big(\sum_{i}\sup_{k_{i+1}\le l\le m\le k_i}|\phi_{2^l}\ast f(x)-\phi_{2^m}\ast f(x)|^2\big)^{1/2}.
\end{equation}
The oscillation inequality follows from the pointwise estimate
\begin{align*}
\mathscr O(\Phi\star f)(x)\le C[S_2(\Phi\star f)(x)+\mathscr O_L(\Phi\star f)(x)].
\end{align*}
Strong $L^p$ estimates and weak $(1,1)$ estimate for $S_2(\Phi\star f)$ have been established as above. For the long oscillation $\mathscr O_L(\Phi\star f)$, we borrow some notations and results from \cite[pp.6724]{JSW08}. For $j\in\mathbb Z$ and $\beta=(m_1,\cdots,m_n)\in\mathbb Z^n$,  we denote the dyadic cube $\prod_{k=1}^n(m_k2^j,(m_k+1)2^j]$ in $\mathbb R^n$ by $Q_\beta^j$, and the set of all dyadic cubes with side length $2^j$ by $\mathcal D_j$. The conditional expectation of a local integrable $f$ with respect to $\mathcal D_j$ is given by
$$
\mathbb E_jf(x)=\sum_{Q\in \mathcal D_j}\frac1{|Q|}\int_{Q}f(y)dy\cdot\chi_{Q}(x)
$$
for all $j\in\mathbb Z$.
Note that $\mathscr O_L$ satisfies
\begin{equation}\nonumber
\mathscr O_L(\Phi\star f)\le \mathscr O_L(\mathscr D f)+\mathscr O_L(\mathscr Ef),
\end{equation}
where
$$
\mathscr Df=\{\phi_{2^k}\ast f-\mathbb E_kf\}_k \quad \text{and}\quad  \mathscr Ef=\{\mathbb E_kf\}_k.
$$
Following inequalities are oscillation inequalities for dyadic martingales (see \cite{JKRW98}),
$$
 |\{x:\mathscr O_L(\mathscr{E}f)(x)>\alpha\}|\le \frac C{\alpha}\|f\|_{L^1}\ \text{and}\ \|\mathscr O_L(\mathscr{E}f)\|_{L^p}\leq C_p\|f\|_{L^p},\ 1<p<\infty.
$$
Next, observe that
\begin{equation}\nonumber
\mathscr O_L(\mathscr D f)\le C\big(\sum_{k\in\mathbb Z}|\phi_k\ast f-\mathbb E_kf|^2\big)^{1/2}:=\mathcal Sf.
\end{equation}
Jones {\it et al} \cite{JSW08} have established the following weak-type $(1,1)$ bound and $L^p$ bounds for $\mathcal S$,
\begin{equation}\label{S1p}
|\{x:\mathcal Sf(x)>\alpha\}|\le \frac C{\alpha}\|f\|_{L^1}\ \text{and}\ \|\mathcal Sf\|_{L^p}\le C_p\|f\|_{L^p},\ \ 1<p<\infty,
\end{equation}
see also \cite{DHL16}.
This completes the proof of Theorem \ref{OSchLp}.
\end{proof}
\begin{remark} With $\phi$ a radial function,
Campbell {\it et al} \cite{CJRW2000,CJRW2002} obtained Theorem \ref{SchLp} and Theorem \ref{OSchLp} by rotation method.
\end{remark}

\section{Variational characterization of $H^p$}
In what follows we shall use the well-known atom decomposition of $H^p$. So, we present the definition of $(p,q)$-atom.
\begin{definition}
Let $0<p\le 1\le q\le\infty$, $p\neq q$. A function $a(x)\in L^q$ is called a $(p,q)$-atom with the center at $x_0$, if it satisfies the following conditions:
\begin{description}
\item (i) Supp $a\subset B(x_0,r)$;
\item (ii) $\|a\|_{L^q}\le |B(x_0,r)|^{\frac1q-\frac1p}$;
\item (iii) $\int a(x)dx=0$.
\end{description}
\end{definition}

\begin{lemma}(\cite{L78})
Let $0<p\le 1$. Given a distribution $f\in H^p$, there exists a sequence of $(p,q)$-atoms  with $1\le q\le\infty$ and $q\neq p$, $\{a_k\}$, and a sequence of scalars $\{\lambda_k\}$ such that
 $$
 f=\sum_{k}\lambda_ka_k\ \ \text{in}\ \ H^p.
 $$
\end{lemma}
\textit{Proof of  Theorem \ref{SchHp}}.\ \ In proving Theorem \ref{SchHp} we consider first the inequality \eqref{SchHp}.
By \cite[Theorem 1.1]{YZ08}, we just need to prove that there exists a positive constant $C$ such that
$\|\mathscr V_\rho(\Phi\star a)\|_{L^p}\le C$ for any $(p,2)$ atom $a$ and $p\in(\frac{n}{n+1},1]$.
\par
Suppose $a$ is supported in a cube $Q$, $x_0$ is the center of $Q$, write $\tilde{Q}=8Q$. By H\"{o}lder's inequality and Theorem \ref{SchLp},
$$
\int_{\tilde{Q}}\mathscr V_\rho(\Phi\star a)^p(x)dx\le |\tilde{Q}|^{1-\frac p2}\|\mathscr V_\rho(\Phi\star a)\|_{L^2}^p\le C|Q|^{1-\frac p2}\|a\|^p_{L^2}\le C.
$$
\par
To deal with $x\in (\tilde{Q})^c$ one uses the cancelation condition of $a$ and Minkowski's inequality,
\begin{align*}
&\mathscr V_\rho(\Phi\star a)(x)\\
&=\sup_{\{\varepsilon_k\}\searrow0}\bigg(\sum_{k}\bigg|\int_{\mathbb R^n}\bigg\{\big[\phi_{\varepsilon_k}(x-y)-\phi_{\varepsilon_{k+1}}(x-y)\big]-\big[\phi_{\varepsilon_k}(x-x_0)-\phi_{\varepsilon_{k+1}}(x-x_0)\big]\bigg\}a(y)dy\bigg|^\rho\bigg)^{\frac1\rho}\\
&\le\int_{Q}|a(y)| \sup_{\{\varepsilon_k\}\searrow0}\bigg(\sum_{k}\bigg|\big[\phi_{\varepsilon_k}(x-y)-\phi_{\varepsilon_k}(x-x_0)\big]-\big[\phi_{\varepsilon_{k+1}}(x-y)-\phi_{\varepsilon_{k+1}}(x-x_0)\big]\bigg|^\rho\bigg)^{\frac1\rho}dy\\
&\le\int_{Q}|a(y)|\|\Phi(x-y)-\Phi(x-x_0)\|_{v_\rho}dy.
\end{align*}
For $x\in (\tilde{Q})^c$, $y\in Q$ and $\theta\in(0,1)$, we have $|x-x_0+\theta (y-x_0)|\sim|x-x_0|$. Note that $\phi(\frac{x-y}t)$ is a smooth function of $t$ on $(0,\infty)$ for any fixed $x,y$.  Applying \eqref{vr1i} and the mean value theorem, we estimate
\begin{align*}
\nonumber\|\Phi(x-y)-\Phi(x-x_0)\|_{v_\rho}\le&\|\Phi(x-y)-\Phi(x-x_0)\|_{v_1}\\
\le& C|y-x_0|\int_0^\infty\frac1{t^{n+2}}\big(1+\frac{|x-x_0|}t)^{-(n+2)}dt\\
\nonumber\le& C\frac{|y-x_0|}{|x-x_0|^{n+1}}\int_0^\infty\frac{t^n}{(1+t)^{n+2}}dt\\
\nonumber\le&C\frac{|y-x_0|}{|x-x_0|^{n+1}}.
\end{align*}
Hence, we obtain the desired bound
\begin{align*}
\int_{(\tilde Q)^c}\mathscr V_\rho(\Phi\star a)^p(x)dx\le C\int_{|x-x_0|\ge 8l(Q)}\frac{l^p(Q)}{|x-x_0|^{(n+1)p}}dx\big(\int_Q|a(y)|dy\big)^p\le C,
\end{align*}
finishing the proof of \eqref{SchHp}.
\par
We now turn to the equivalence of two conditions in Theorem \ref{SchHp}. Note that $M_\phi f(x)\le |\phi\ast f(x)|+\mathscr V_\rho(\Phi\star f)(x)$. Thus, $(ii)$ implies $(i)$. Conversely, $M_\phi f\in L^p$ means $f\in H^p$ for $\frac n{n+1}<p\le 1$ and $f\in L^p$ for $1<p<\infty$. Also, the pointwise estimate $\phi\ast f(x)\le M_\phi f(x)$ shows $\phi\ast f\in L^p$ for $\frac n{n+1}<p<\infty$. From \eqref{SchHp} and \eqref{VISchp}, we obtain $\mathscr V_\rho(\Phi\star f)\in L^p$, finishing the proof of Theorem \ref{SchHp}.
\section{Oscillation on $H^p$}
 Estimates for the oscillation operator acting on $H^p$ are trivial. It suffices to show that there exists a positive constant $C$ such that
$\|\mathscr O(\Phi\star a)\|_{L^p}\le C$ for any $(p,2)$ atom $a$ and $p\in(\frac{n}{n+1},1]$. By Theorem \ref{OSchLp} and H\"{o}lder's inequality, we have
$$
\int_{\tilde{Q}}\mathscr O(\Phi\star a)^p(x)dx\le C|Q|^{1-\frac p2}\|\mathscr O(\Phi\star a)\|_{L^2}^p\le C|Q|^{1-\frac p2}\|a\|_{L^2}^p\le C.
$$
Next we consider the integral of $\mathscr O(\Phi\star a)^p$ on $(\tilde{Q})^c$.  Note that $\mathscr O(\Phi\star a)(x)\le C\mathscr V_2(\Phi\star a)(x)$. Applying the same argument, we obtain
\begin{equation}\nonumber
\int_{(\tilde{Q})^c}\mathscr O(\Phi\star a)^p(x)dx\le C\int_{(\tilde{Q})^c}\mathscr V_2(\Phi\star a)^p(x)dx\le C\int_{|x-x_0|\ge 8l(Q)}\frac{l^p(Q)}{|x-x_0|^{(n+1)p}}\big(\int_Q|a(y)|dy\big)^pdx \le C.
\end{equation}
Hence, we get $\|\mathscr O(\Phi\star f)\|_{L^p}\le C\|f\|_{H^p}$.
\par
We now turn to the negative result in Theorem \ref{oHp}. Let $\phi(x)=f(x)=e^{-x^2}$ for $x\in\mathbb R$, $\phi_t\ast f(x)=\frac1{\sqrt{t^2+1}}e^{-\frac{x^2}{1+t^2}}$ , $\phi\ast f(x)=\frac{\sqrt2}{2}e^{-\frac{x^2}{2}}\in L^p$ for any $p\in(0,1]$. Define $F(s,x)=se^{\frac{x^2}{s^2}}$ for $(s,x)\in(1,+\infty)\times \mathbb R^+$. Clearly, for fixed $x\in[0,\frac{\sqrt{2}}2]$, $F(s,x)$ is increasing respect to $s$; for fixed $x\in(\frac{\sqrt{2}}2,+\infty)$, $F(s,x)$ is decreasing on $(1,\sqrt2x]$ and increasing on $(\sqrt2x,+\infty)$. Note that $M_\phi(f)$ is even. Therefore,
\begin{equation}\nonumber
M_\phi(f)(x)=
\begin{cases}
e^{-x^2},&\mbox{$x\in[-\frac{\sqrt{2}}2,\frac{\sqrt{2}}2]$,}\\
\frac1{\sqrt{2e}}\frac1x, &\mbox{ $x\in(-\infty,-\frac{\sqrt{2}}2)\bigcup(\frac{\sqrt{2}}2,+\infty)$.}
\end{cases}
\end{equation}
Obviously, $M_\phi(f)\notin L^p$ for any $p\in(0,1]$.
\par
For the oscillation of $\Phi\star f$, we take $t_n=\frac1n$ and use the following pointwise estimate
\begin{align*}
\mathscr O(\Phi\star f)(x)&\le \sum_n\sup_{\frac1{n+1}\le\varepsilon_{n+1}<\varepsilon_n\le \frac1n}|\phi_{\varepsilon_{n+1}}\ast f(x)-\phi_{\varepsilon_{n}}\ast f(x)|.
\end{align*}
For fixed $x\in[0,\frac{\sqrt2}2]$, $\frac 1se^{-\frac{x^2}{s^2}}$ is decreasing on $(1,+\infty)$. So,
\begin{align*}
\mathscr O(\Phi\star f)(x)&\le \sum_n|\phi_{\frac1{n+1}}\ast f(x)-\phi_{{\frac1n}}\ast f(x)|\\
&=\phi_0\ast f(x)-\phi_1\ast f(x)=e^{-x^2}-\frac{\sqrt2}2e^{-\frac{x^2}2}.
\end{align*}
For fixed $x\in(\frac{\sqrt2}2,1)$, $\frac1se^{-\frac{x^2}{s^2}}$ is increasing on $(1,\sqrt2x]$ and decreasing on $(\sqrt2x,\sqrt2]$. We estimate
\begin{align*}
\mathscr O(\Phi\star f)(x)&\le \sum_{\frac1{n+1}\le \sqrt{2x^2-1}}\sup_{\frac1{n+1}\le\varepsilon_{n+1}<\varepsilon_n\le \frac1n}|\phi_{\varepsilon_{n+1}}\ast f(x)-\phi_{\varepsilon_{n}}\ast f(x)|\\
&+\sum_{\frac1n> \sqrt{2x^2-1}}\sup_{\frac1{n+1}\le\varepsilon_{n+1}<\varepsilon_n\le \frac1n}|\phi_{\varepsilon_{n+1}}\ast f(x)-\phi_{\varepsilon_{n}}\ast f(x)|\\
&\le \sqrt{\frac2e}\frac1x-e^{-x^2}-\frac{\sqrt2}2e^{-\frac{x^2}2}.
\end{align*}
For fixed $x\in[1,+\infty)$, $\frac1se^{-\frac{x^2}{s^2}}$ is increasing on $(1,\sqrt2]$. Hence
\begin{align*}
\mathscr O(\Phi\star f)(x)&\le \sum_n|\phi_{\frac1{n+1}}\ast f(x)-\phi_{{\frac1n}}\ast f(x)|\\
&=\phi_1\ast f(x)-\phi_0\ast f(x)=\frac{\sqrt2}2e^{-\frac{x^2}2}-e^{-x^2}.
\end{align*}
Obviously,$\mathscr O(\Phi\star f)$ is even and $\mathscr O(\Phi\star f)\in L^p$ for any $p\in(0,1]$, completing the proof of Theorem \ref{oHp}.
\section{$\lambda$-jump on $H^p$}
\textit{Proof of  Theorem \ref{ljrhp}}.\ \
It is clear that for $\lambda>0$ and $\rho\ge1$
 \begin{equation}\nonumber
\lambda[N_\lambda(\Phi\star f)(x)]^{1/\rho}\leq C_\rho\mathscr V_\rho(\Phi\star f)(x).
\end{equation}
Consequently, we have
\begin{equation}\nonumber
\|\lambda[N_\lambda(\Phi\star f)]^{1/\rho}\|_{L^p}\leq C_\rho\|\mathscr V_\rho(\Phi\star f)\|_{L^p}\le C_\rho\|f\|_{H^p},\ \ p\in(\frac{n}{n+1},1],
\end{equation}
uniformly in $\lambda>0$.
\par
For counterexample, we take $\phi(x)=f(x)=e^{-x^2}$. Obviously, $f\notin H^p$ and $M_\phi(f)\notin L^p$ for any $p\in(0,1]$.
\par
 When $\lambda\ge1$, we have $N_\lambda(\Phi\star f)(x)\equiv0$ and $\|\lambda[N_\lambda(\Phi\star f)]^{1/\rho}\|_{L^p}<\infty$ uniformly in $\lambda>0$ for any $p\in(0,1]$.
 \par
 When
 $\frac1{\sqrt e}\le\lambda<1$, we get $N_\lambda(\Phi\star f)(x)\le e^{-x^2}\lambda^{-1}$ for $|x|\le \sqrt{-\ln\lambda}$ and $N_\lambda(\Phi\star f)(x)=0$ for $\sqrt{-\ln\lambda}<x$. Hence, $\|\lambda[N_\lambda(\Phi\star f)]^{1/\rho}\|_{L^p}<\infty$ uniformly in $\lambda>0$ for any $p\in(0,1]$.
 \par
 When $0<\lambda<\frac1{\sqrt e}$, we obtain
 \begin{equation}\nonumber
N_\lambda(\Phi\star f)(x)\le C
\begin{cases}
e^{-x^2}\lambda^{-1},&\mbox{$x\in[-\frac{\sqrt{2}}2,\frac{\sqrt{2}}2]$,}\\
\sqrt{\frac 2e}\frac1{\lambda |x|}- e^{-x^2}\lambda^{-1}, &\mbox{ $x\in(-\frac1{\lambda\sqrt{2e}},-\frac{\sqrt{2}}2)\bigcup(\frac{\sqrt{2}}2,\frac1{\lambda\sqrt{2e}})$,}\\
0,&\mbox{ $x\in(-\infty,-\frac1{\lambda\sqrt{2e}})\bigcup(\frac1{\lambda\sqrt{2e}},+\infty)$.}
\end{cases}
\end{equation}
One can establish the following $L^p$ bounds:
\begin{align*}
\|\lambda[N_\lambda(\Phi\star f)]^{1/\rho}\|_p^p&\le C \lambda^{p(1-1/\rho)}\int_0^{\frac{\sqrt2}2}e^{-px^2/\rho}dx+C\lambda^{p(1-1/\rho)}\int_{\frac{\sqrt2}2}^{\frac1{\lambda\sqrt{2e}}}x^{-\frac p{\rho}}dx\\
&\le C+C\lambda^{p-1}.
\end{align*}
Consequently, $\|\lambda[N_\lambda(\Phi\star f)]^{1/\rho}\|_{L^1}<\infty$ uniformly in $\lambda>0$ for $\rho\in(1,\infty)$.\qed
\par
Theorems 1.4 suggests the following improvement:
\begin{conjecture}
For $\frac n{n+1}<p\le 1$, there exists $C_{p}>0$ such that
\begin{equation}\nonumber
\|\lambda\sqrt{N_\lambda(\Phi\star f)}\|_{L^p}\le C_{p}\|f\|_{H^p},
\end{equation}
uniformly in $\lambda>0$.
\end{conjecture}
In the case of analogous for variation, oscillation and $\lambda$-jump operators, we know the conjecture above is possible. However, our current techniques do not allow us to prove it.

\end{document}